\documentclass[11pt,a4paper]{article}
\usepackage[T1]{fontenc}
\usepackage[utf8]{inputenc}
\usepackage{amsmath,amssymb,amsthm}
\usepackage{xcolor}
\usepackage{tikz-cd}
\usepackage[colorlinks=true,linkcolor=blue,citecolor=blue]{hyperref}

\title{A new characterization of arithmetical categories and Pixley's theorem
}
\author{Marino Gran\thanks{This research was supported by the Fonds de la Recherche
Scientifique - FNRS under Grant CDR No.J.0092.26. 
The author made use of Claude Fable 5 (Anthropic) in the
preparation of this paper: the AI assistant reorganised the structure
of the paper, produced all the commutative diagrams, and improved the
overall presentation and exposition. The author verified all AI-assisted material
and takes full responsibility for the final text.
}}
\date{}

\theoremstyle{plain}
\newtheorem{theoreme}{Theorem}[section]
\newtheorem{proposition}[theoreme]{Proposition}
\newtheorem{lemme}[theoreme]{Lemma}
\newtheorem{corollaire}[theoreme]{Corollary}
\theoremstyle{definition}
\newtheorem{definition}[theoreme]{Definition}

\theoremstyle{remark}
\newtheorem{remarque}[theoreme]{Remark}
\theoremstyle{plain}
\newtheorem*{theoremA}{Theorem A}

\newcommand{\C}{\mathcal{C}}
\newcommand{\Set}{\mathbf{Set}}
\newcommand{\Equiv}{\operatorname{Equiv}}
\newcommand{\Eq}{\operatorname{Eq}}
\newcommand{\im}{\operatorname{Im}}
\newcommand{\md}[1]{\ (\mathrm{mod}\ #1)}
\newcommand{\op}{\mathrm{o}}
\newcommand{\finaladdress}[3]{%
  \bigskip\noindent
  {\scshape #1}\par
  \smallskip\noindent
  {\scshape #2}\par
  \smallskip\noindent
  {\itshape Email address}: \texttt{#3}\par}
\begin{document}
\maketitle

\begin{abstract}
We characterize the arithmetical categories - the
exact Mal'tsev categories whose lattices of congruences are
distributive - by a simple condition on 
some suitable pushouts and a finite limit: a regular category is exact arithmetical if and only if, for every triple of regular epimorphisms with
common domain, the pairwise pushouts exist and the comparison morphism to
the limit of the diagram they form is a regular epimorphism. For a regular category, the validity of
this condition for pairs and for triples is equivalent to its validity for
$n$-tuples for all $n \ge 2$, and characterizes arithmetical categories
among regular ones. For pairs alone the condition reduces to the
characterization of exact Mal'tsev categories among regular ones, due to
Carboni, Kelly and Pedicchio. For $n$-tuples with
$n \ge 4$ nothing further is obtained: arithmeticity is the third and last
rung of that ladder. The proof rests on a categorical form of the \emph{Chinese
Remainder Theorem}; in the exact Mal'tsev context this form
of the theorem lies close to results of Hoefnagel on majority categories. As a direct
application, we give a new proof of Pixley's characterization of
arithmetical varieties using a suitable diagram of free algebras. \end{abstract}


\medskip


\section{Introduction}

A variety of universal algebras is \emph{arithmetical} when its congruences
are both permutable and distributive. Pixley \cite{Pixley} characterized
these varieties by the existence of a single ternary term $p$ satisfying
the identities
\[
p(x,x,z) = z, \qquad p(x,y,x) = x, \qquad p(x,y,y) = x,
\]
which combines a Mal'tsev term and a majority term, and by the validity of
the \emph{Chinese Remainder Theorem} (CRT): a finite system of congruences
$x \equiv a_i \md{\theta_i}$ admits a solution as soon as it is pairwise
compatible, i.e.\ $a_i \equiv a_j \md{\theta_i \vee \theta_j}$ for all
$i, j$ --- the form the classical theorem takes in the ring of integers $\mathbb{Z}$, with no
comaximality hypothesis. Boolean
algebras, Heyting algebras, von Neumann regular rings \cite{GRos}, MV-algebras, residuated lattices \cite{GJKO}, any discriminator
variety \cite{BS}, are all arithmetical varieties, while groups, rings and Lie algebras are Mal'tsev varieties which are not arithmetical.

In categorical algebra, the permutability of congruences is captured by the
notion of a Mal'tsev category
\cite{CLP, CPP}: a finitely complete category in which every (internal) reflexive
relation is an equivalence relation. In the regular context
\cite{Barr}, Carboni, Kelly and Pedicchio proved a remarkable
characterization \cite[Theorem~5.7]{CKP}: a regular category is an
exact Mal'tsev category if and only if, for every pair of regular
epimorphisms $r$ and $s$ with the same domain
\begin{equation}\label{pushout}
\begin{tikzcd}[column sep=3.4em, row sep=3em]
A
\arrow[drr, two heads, bend left=12, "r"]
\arrow[ddr, two heads, bend right=28, "s"']
\arrow[dr, two heads, dashed, "w" description] & & \\
& B \times_D C
\arrow[r, "\pi_B"]
\arrow[d, "\pi_C"']
\arrow[dr, phantom, pos=0.12, "\lrcorner"] &
B \arrow[d, two heads, "u"] \\
& C \arrow[r, two heads, "v"'] & D
\end{tikzcd}
\end{equation}
 their pushout $(D,u,v)$ exists and the comparison
morphism $w$ from the domain to the pullback of the pushout is a regular
epimorphism.

Arithmetical categories were introduced by Pedicchio
\cite{Pedicchio} as (Barr-)exact Mal'tsev categories with coequalizers whose
congruence lattices are distributive. Bourn \cite{Bourn} subsequently
showed that the coequalizer assumption was not needed, and he characterized
arithmetical categories, among exact Mal'tsev ones, by the property that
every internal groupoid is an equivalence relation.
Among the interesting examples of arithmetical categories which are not varietal we mention the dual of any elementary topos \cite{CKP, Pedicchio, Bourn3}.
A relevant part of the congruence-theoretic content presented below is already in
the literature. Indeed, Hoefnagel
\cite{Hoefnagel, Hoefnagel2, Hoefnagel3} introduced \emph{majority categories}, the
categorical counterpart of varieties admitting a majority term, and
established for arbitrary regular categories a ``Pairwise Chinese Remainder
Theorem'' (PCRT): a system $x \equiv a_i \md{\theta_i}$ whose equations are
solvable two at a time --- in an approximate sense, up to precomposition
with a regular epimorphism --- is solvable. His main theorem shows that the
PCRT characterizes regular majority categories, alongside Pixley's
inclusion
$\alpha \wedge (\beta \circ \gamma) \le
(\alpha \wedge \beta) \circ (\alpha \wedge \gamma)$
for reflexive relations and a categorical form of Bergman's
double-projection theorem \cite{Bergman}; and a regular Mal'tsev category
is congruence-distributive precisely when it is a majority category
\cite[Cor.~2]{Hoefnagel}. Combining these two statements already yields, in
the exact Mal'tsev case, a Chinese-remainder characterization of
arithmeticity, so that Theorem~\ref{thm:principal} below is close to a
reformulation of \cite{Hoefnagel}. What exactness adds is that pairwise
solvability becomes pairwise \emph{compatibility}
(Proposition~\ref{prop:binaire}) and that approximate solutions become
genuine ones. The statement can then be presented as the simple property that a
canonical comparison morphism into a finite limit of quotients is itself a quotient.


This paper has two aims. The first is to explain that the \emph{Chinese Remainder Theorem}
and the arithmeticity of a category are a simple property of a special type of \emph{diagrams} involving some pushouts and a finite limit
(Theorems~\ref{thm:main} and~\ref{cor:diag} below) that can be seen as a $3$-dimensional version of the $2$-dimensional one used by Carboni, Kelly and Pedicchio to characterize exact Mal'tsev categories:

\begin{theoremA}
A regular category $\mathcal{A}$ is an exact arithmetical 
category if and only if, for every triple of regular epimorphisms
$r_1, r_2, r_3$ with common domain, the pairwise pushouts $D_{ij}$ of
$(r_i, r_j)$ exist and, in the diagram
\begin{equation}\label{cube}
\begin{tikzcd}[column sep=1.6em, row sep=2.2em]
& & X
\arrow[dll, two heads, "r_1"']
\arrow[d, two heads, "r_2"]
\arrow[drr, two heads, "r_3"'{pos=0.65}]
\arrow[ddd, dashed, two heads, "a",
to path={(\tikztostart.east) .. controls +(5.2,-1.9) and +(5.2,1.9) ..
(\tikztotarget.east)\tikztonodes}] & & \\
B_1
\arrow[d, two heads]
\arrow[drr, two heads] & &
B_2
\arrow[dll, two heads, crossing over]
\arrow[drr, two heads, crossing over] & &
B_3
\arrow[dll, two heads]
\arrow[d, two heads] \\
D_{12} & &
D_{13} & &
D_{23} \\
& & L
\arrow[ull]
\arrow[u]
\arrow[urr] & &
\end{tikzcd}
\end{equation}
where $L$ denotes the limit of the two middle rows, the canonical
comparison morphism $a \colon X \rightarrow L$ is a regular epimorphism.
\end{theoremA}

For pairs instead of triples, the limit in question is the pullback of the
pushout as in \eqref{pushout}, and the condition is exactly that of \cite[Theorem~5.7]{CKP}: the
characterization places arithmeticity one rung above the Mal'tsev property
on a single ladder. This will be deduced from the main theorem of the
paper (Theorem~\ref{thm:main}), which characterizes arithmetical
categories among regular ones: a regular category is an exact arithmetical category if and only if it satisfies the displayed condition for
pairs and for triples of regular epimorphisms, if and only if it satisfies
it for $n$-tuples, for every $n \ge 2$. In particular the ladder stops at
its third rung: for $n$-tuples with $n \ge 4$
the condition is equivalent to the condition for triples
(Corollary~\ref{cor:nary}). The heart of the proof is a categorical form
of the CRT for congruences (Propositions~\ref{prop:arith-crt}
and~\ref{prop:crt-arith}), from which we also deduce the
congruence-theoretic characterization (Theorem~\ref{thm:principal}): an
exact Mal'tsev category is arithmetical
if and only if every pairwise compatible finite family of congruences is
solvable, in the sense made precise by Definition~\ref{def:trc}. Theorem~A
then sharpens the main theorem by removing the condition on pairs: the
passage from triples down to pairs uses the image factorization
of the terminal morphism.

The second aim is a categorical proof of Pixley's theorem
(Theorem~\ref{thm:pixley}). Instantiating the generic diagram \eqref{cube} at the free
algebra $X+X+X$ on three generators \eqref{Free}, with the three congruences each
identifying different pairs of generators, the term $p$ arises from a single
application of the condition, the required lifting being provided by the
projectivity of free algebras with respect to regular epimorphisms. This new
proof is inspired by the method of \cite{CP} for Mal'tsev varieties, and of \cite{GR} for $3$-permutable varieties.

All proofs are essentially carried out using the calculus of relations of a regular
category, in the spirit of \cite{CLP, CKP}, with no recourse to generalized
elements or embedding theorems. Apart from Theorem~\ref{thm:ckp57}(a),
which is quoted from \cite{CKP}, the paper is essentially self-contained. 

\section{Preliminaries}

\subsection{Regular and exact categories}

All categories will be assumed to be finitely complete. A category is
\emph{regular} if kernel pairs admit coequalizers and regular epimorphisms
are stable under pullbacks. Equivalently, every morphism factors as a regular epimorphism followed by a monomorphism,
its \emph{image}, and these factorisations are pullback-stable. It is \emph{exact} (in the sense of Barr \cite{Barr}) if,
moreover, every internal equivalence relation is \emph{effective}, that is,
is the kernel pair of some morphism.

In a regular category, a \emph{relation} $R$ from $X$ to $Y$ is a subobject
$\langle r_1, r_2\rangle\colon R \rightarrowtail X \times Y$; its
\emph{opposite} relation $R^{\op}$ is the subobject
$\langle r_2, r_1\rangle\colon R \rightarrowtail Y \times X$. The composite relation $S \circ R$
of $R$, from $X$ to $Y$, and $S$, from $Y$ to $Z$, is constructed by forming
the pullback of $r_2$ and $s_1$ and then the image of the induced morphism
to $X \times Z$. 
The composition of relations is associative and monotone, and
$(S \circ R)^{\op} = R^{\op} \circ S^{\op}$ (see
\cite[\S 2]{CKP}, for instance). 
Every morphism $f\colon X \to Y$ is
viewed as a relation via its graph $\langle 1_X, f\rangle \colon X \to X \times Y$. We'll identify a morphism $f$ with the corresponding relation $\langle 1_X, f\rangle$, and the composition of relations defined above extends the usual composition of morphisms under this convention.

An (internal) equivalence relation
$\theta \rightarrowtail X \times X$ is a relation on $X$ which is reflexive, symmetric and transitive. The poset $\Equiv(X)$ of equivalence relations on
$X$ admits finite meets, given by intersection of subobjects of
$X \times X$; its minimum is the discrete relation $\Delta_X$, its maximum the indiscrete
relation $\nabla_X = X \times X$. In an \emph{exact} category, every
equivalence relation $\theta$ admits a quotient, a regular epimorphism
$q_\theta \colon X \twoheadrightarrow X/\theta$ whose kernel pair is
$\theta$. For $f \colon X \to Y$, we write $\Eq(f) \in \Equiv(X)$ for its
kernel pair.

The next lemma will play a central role in this article: each clause of it is a
direct computation with limits and images (see \cite[\S 2]{CKP}, or \cite{Gran}).

\begin{lemme}\label{lem:calcul}
In a regular category, let $f\colon X \to B$, $g\colon X \to C$,
$h\colon X' \to C$, $u\colon B \to D$, $v\colon C \to D$ be morphisms,
viewed as relations, and let $T$ be a relation from $B$ to $C$.
\begin{itemize}
\item[(1)] $\Delta_X \le f^{\op} \circ f = \Eq(f)$, with equality
$f^{\op} \circ f = \Delta_X$ if and only if $f$ is a monomorphism.
\item[(2)] $f \circ f^{\op} \le \Delta_B$, with equality if and only if $f$
is a regular epimorphism.
\item[(3)] $v^{\op} \circ u$ is the subobject
$\langle p_B, p_C\rangle \colon B \times_D C \rightarrowtail B \times C$
determined by the pullback of $u$ and $v$; $g \circ f^{\op} = \im \langle f,g\rangle$, where $\im \langle f,g\rangle$ denotes the image of $\langle f,g\rangle \colon X \rightarrow B \times C$; and
$\Eq(\langle f,g\rangle) = \Eq(f) \wedge \Eq(g)$.
\item[(4)] $h^{\op} \circ T \circ f = (f \times h)^{-1}(T)$; in particular
the operation $T \mapsto h^{\op} \circ T \circ f$ commutes with
finite intersections.
\end{itemize}
\end{lemme}

\begin{proof}
(1) The composite $f^{\op} \circ f$ is computed by the pullback
of $f$ along $f$.

(2) The composite $f \circ f^{\op}$ is the image of
$\langle f,f\rangle = \Delta_B \circ f \colon X \to B \times B$. Let
$f = m \circ e$ be the (regular epi, mono) factorization, with
$I := \im (f)$. Then
$\langle f, f\rangle = (m \times m) \circ \Delta_I \circ e$, where
$\Delta_I \circ e$ is a regular epimorphism followed by a monomorphism, and
$m \times m$ is a monomorphism.  The image is therefore 
$\Delta_I$, viewed in $B \times B$. It equals $\Delta_B$ if and only if $m$
is an isomorphism.

(3) The composite $v^{\op} \circ u$ is computed by the pullback
of $u$ and $v$, and the induced morphism to $B \times C$ is the canonical
monomorphism of the pullback. The composite $g \circ f^{\op}$ is,
by definition, the image of $\langle f,g\rangle$. The equality $\Eq(\langle f,g\rangle) = \Eq(f) \wedge \Eq(g)$ is an exercise on pullbacks.

(4) Composing a relation with the graph of a morphism on the right, or with the
opposite of a graph on the left, involves only pullbacks, not regular images. 
The graph $h^{\op} \circ T \circ f$ is then a subobject of $X \times X'$, which coincides with the
subobject
$(f\times h)^{-1}(T)$, as desired.
\end{proof}

Two consequences of these results will be used. First, for an effective equivalence relation $\theta$ with
quotient $q_\theta$:
\[
q_\theta^{\op} \circ q_\theta = \theta,
\qquad
q_\theta \circ q_\theta^{\op} = \Delta_{X/\theta}.
\]
Next, if $\mu \le \theta$ and $s \colon X/\mu \to X/\theta$ is the canonical
morphism such that $s \circ q_\mu = q_\theta$, then
$s = s \circ (q_\mu \circ q_\mu^{\op})
= (s \circ q_\mu) \circ q_\mu^{\op}
= q_\theta \circ q_\mu^{\op}$. Finally, an elementary \emph{absorption}: if $\rho, \sigma \le
\alpha$ are equivalence relations on $X$, then
$\rho \circ \alpha \circ \sigma = \alpha$, since
$\Delta \circ \alpha \circ \Delta \le \rho \circ \alpha \circ \sigma
\le \alpha \circ \alpha \circ \alpha = \alpha$.

\subsection{Mal'tsev categories}

\begin{definition}
A finitely complete category is a \emph{Mal'tsev category} if every reflexive internal
relation $R \rightarrowtail X \times X$ in it is an equivalence relation.
\end{definition}

\begin{theoreme}[{\cite{CKP}}]\label{thm:ckp}
For a regular category $\C$, the following conditions are equivalent:
(a) $\C$ is a Mal'tsev category;
(b) for every object $X$ and all equivalence relations $\theta, \varphi$ on $X$, one
has $\theta \circ \varphi = \varphi \circ \theta$.
\end{theoreme}

In a variety of universal algebras, internal equivalence relations in the categorical
sense coincide with the usual congruences, and condition (b) is classical
congruence-permutability, equivalent to the existence of a ternary term $p$
satisfying $p(x,y,y) = x$ and $p(x,x,y) = y$ (the classical Mal'tsev theorem \cite{Malcev}).
Groups, quasi-groups, loops, rings, Lie algebras, Heyting algebras are Mal'tsev varieties, whereas sets, monoids and
lattices are not. Among the examples of Mal'tsev categories which are not varieties let us mention the dual category of any elementary topos, and the categories of topological groups and of cocommutative Hopf algebras over a field.

\begin{lemme}\label{lem:sup}\cite{CLP}
Let $\C$ be an exact Mal'tsev category and
$\theta, \varphi \in \Equiv(X)$. Then the composite relation $\theta \circ \varphi$
is an equivalence relation, and it is the join of $\theta$ and $\varphi$ in
$\Equiv(X)$:
\[
\theta \vee \varphi = \theta \circ \varphi = \varphi \circ \theta .
\]
Moreover, the lattice $\Equiv(X)$ is modular \cite{CKP}.
\end{lemme}


\section{The Chinese Remainder Theorem}

We now assume that $\C$ is an exact category in which pushouts of regular epimorphisms along regular epimorphisms exist.  Let $X$ be an object of $\C$ and
$\theta_1, \dots, \theta_n \in \Equiv(X)$, $n \ge 2$. Write
$q_i \colon X \twoheadrightarrow X/\theta_i$ and, for $i \ne j$,
$q_{ij} \colon X/\theta_i \twoheadrightarrow X/(\theta_i \vee \theta_j)$ for
the canonical quotients, where the join $\theta_i \vee \theta_j$ exists since it is the kernel pair of the diagonal of the pushout of $q_i$ along $q_j$. 

\begin{definition}\label{def:trc}
Let $L(\theta_1,\dots,\theta_n)$ be the limit of the diagram formed by the
objects $X/\theta_i$ ($1 \le i \le n$) and $X/(\theta_i \vee \theta_j)$
($1 \le i < j \le n$), together with the morphisms $q_{ij}$; it is the
subobject of $\prod_i X/\theta_i$ of \emph{pairwise compatible} families.
The evident cone $(q_i)_i$ induces a canonical morphism
$c \colon X \to L(\theta_1,\dots,\theta_n)$, which factors through a comparison
morphism
\[
\Phi \colon X\Bigl/\bigwedge_{i=1}^n \theta_i
\longrightarrow L(\theta_1,\dots,\theta_n).
\]
We say that the family $(\theta_1,\dots,\theta_n)$ \emph{satisfies the Chinese Remainder Theorem CRT}
if the morphism $c \colon X \to L(\theta_1,\dots,\theta_n)$ is a regular epimorphism. This is equivalent to asking that the comparison morphism $\Phi$ is an isomorphism (Lemma~\ref{lem:mono} below).

In general, we'll say that $\C$ \emph{satisfies}
$(\mathrm{CRT}_n)$ if every family of $n$ equivalence relations on every object of
$\C$ does.
\end{definition}

For $n = 2$ the diagram reduces to
$X/\theta_1 \to X/(\theta_1 \vee \theta_2) \leftarrow X/\theta_2$ and
$L(\theta_1,\theta_2)$ is the corresponding pullback. For $n = 3$, the
relevant diagram is
\begin{equation}\label{cube2}
\begin{tikzcd}[column sep=1.2em, row sep=2.2em]
& & X
\arrow[dll, two heads, "q_1"']
\arrow[d, two heads, "q_2"]
\arrow[drr, two heads, "q_3"'{pos=0.65}]
\arrow[ddd, dashed, two heads, "c",
to path={(\tikztostart.east) .. controls +(6.2,-1.9) and +(6.2,1.9) ..
(\tikztotarget.east)\tikztonodes}] & & \\
X/\theta_1
\arrow[d, two heads]
\arrow[drr, two heads] & &
X/\theta_2
\arrow[dll, two heads, crossing over]
\arrow[drr, two heads, crossing over] & &
X/\theta_3
\arrow[dll, two heads]
\arrow[d, two heads] \\
X/(\theta_1 \vee \theta_2) & &
X/(\theta_1 \vee \theta_3) & &
X/(\theta_2 \vee \theta_3) \\
& & L(\theta_1, \theta_2, \theta_3)
\arrow[ull]
\arrow[u]
\arrow[urr] & &
\end{tikzcd}
\end{equation}
where $L(\theta_1,\theta_2,\theta_3)$ is the limit of the two middle rows
and the dashed morphism $c$ is the canonical one, induced by the cone
$(q_1, q_2, q_3)$. Then the Chinese Remainder Theorem $(\mathrm{CRT}_3)$ asserts that $c$ is a regular
epimorphism --- equivalently (Lemma~\ref{lem:mono} below), that the
comparison
$\Phi \colon X/(\theta_1 \wedge \theta_2 \wedge \theta_3) \to
L(\theta_1,\theta_2,\theta_3)$ is an isomorphism. In a variety,
$L(\theta_1,\dots,\theta_n)$ is the algebra of pairwise compatible families
of classes, and $(\mathrm{CRT}_n)$ says that every system
$x \equiv a_i \md{\theta_i}$ satisfying the compatibilities
$a_i \equiv a_j \md{\theta_i \vee \theta_j}$ admits a solution, unique
modulo $\bigwedge_i \theta_i$.

\begin{lemme}\label{lem:mono}
Let $\theta_1, \dots, \theta_n$ be equivalence relations on $X$ in an exact category $\mathcal C$. The comparison
morphism $\Phi \colon X/\bigwedge_i \theta_i \to L(\theta_1,\dots,\theta_n)$
is always a monomorphism. Consequently, $\Phi$ is an isomorphism if and only
if the canonical morphism $c \colon X \to L(\theta_1,\dots,\theta_n)$ is a
regular epimorphism.
\end{lemme}

\begin{proof}
We first check that the kernel pair of $c$ is $\bigwedge_i \theta_i$. The
projections from the limit to the $X/\theta_i$ are jointly monic, hence
$m \colon L(\theta_1,\dots,\theta_n) \rightarrowtail
\prod_i X/\theta_i$ with $m \circ c = \langle q_i\rangle_i$ is a monomorphism. By
Lemma~\ref{lem:calcul}(1)(3),
\[
\Eq(c) \;=\; \Eq(m \circ c) \;=\; \bigwedge_i \Eq(q_i)
\;=\; \bigwedge_i \theta_i \;=:\; \kappa .
\]
The morphism $c$
factors as $c = \Phi \circ q$, where $q \colon X \rightarrow X/ \kappa$ is a regular epimorphism and
$\Eq(q) = \kappa = \Eq(c)$. One easily checks that $\Phi^{\op} \circ \Phi =\Delta$,
and $\Phi$ is a monomorphism by Lemma~\ref{lem:calcul}(1).

Moreover, since $q$ is a regular epimorphism and $c = \Phi \circ q$, $\Phi$ is an isomorphism exactly when $c$ is a regular
epimorphism.
\end{proof}

The following result is due to Carboni, Kelly and
Pedicchio:

\begin{theoreme}[{\cite[Theorems 5.2 and 5.7]{CKP}}]\label{thm:ckp57}
Let $\mathcal{A}$ be a regular category.
\begin{itemize}
\item[(a)] Let $u \circ r = v \circ s =: t$ be a commutative square of regular
epimorphisms with common domain $A$, with kernel pairs $R = \Eq(r)$,
$S = \Eq(s)$, $T = \Eq(t)$, and let $w$ be the comparison morphism from $A$
to the pullback of $u$ and $v$, as in the diagram
\[
\begin{tikzcd}[column sep=3.4em, row sep=3em]
A
\arrow[drr, two heads, bend left=12, "r"]
\arrow[ddr, two heads, bend right=28, "s"']
\arrow[dr, dashed, "w" description] & & \\
& B \times_D C
\arrow[r, "\pi_B"]
\arrow[d, "\pi_C"']
\arrow[dr, phantom, pos=0.12, "\lrcorner"] &
B \arrow[d, two heads, "u"] \\
& C \arrow[r, two heads, "v"'] & D
\end{tikzcd}
\]
Then $w$ is a regular epimorphism if and only if
$R \circ S = T$, if and only if $S \circ R = T$; the outer square is then a
pushout, and it is a
pullback if and only if moreover $R \wedge S = \Delta$.
\item[(b)] $\mathcal{A}$ is an exact Mal'tsev category if and only if, for
all regular epimorphisms $r \colon A \to B$ and $s \colon A \to C$ with a
common domain, the pushout of $r$ and $s$ exists and the comparison morphism
$w$ from $A$ to the pullback of this pushout is a regular epimorphism.
\end{itemize}
\end{theoreme}
From this, we deduce the following
\begin{proposition}\label{prop:binaire}
Every exact Mal'tsev category satisfies $(\mathrm{CRT}_2)$: for all
congruences $\theta, \varphi$ on $X$, the square
\[
\begin{tikzcd}[column sep=3.6em, row sep=2.6em]
X/(\theta \wedge \varphi)
\arrow[r, two heads, "\bar v"]
\arrow[d, two heads, "\bar u"'] &
X/\theta
\arrow[d, two heads, "u"] \\
X/\varphi
\arrow[r, two heads, "v"'] &
X/(\theta \vee \varphi)
\end{tikzcd}
\]
of canonical morphisms - so that
$u \circ q_\theta = v \circ q_\varphi = q_{\theta\vee\varphi}$ and
$\bar v \circ q_{\theta\wedge\varphi} = q_\theta$,
$\bar u \circ q_{\theta\wedge\varphi} = q_\varphi$ - is a pullback.
Equivalently, in the calculus of relations,
\[
q_\varphi \circ q_\theta^{\op}
\;=\; v^{\op} \circ u .
\]
Moreover the square $u \circ q_\theta = v \circ q_\varphi$, with vertex
$X/(\theta\vee\varphi)$, is a pushout of $q_\theta$ and $q_\varphi$.
\end{proposition}

\begin{proof}
By Theorem~\ref{thm:ckp57}(b) the pushout of $q_\theta$ and $q_\varphi$
exists and the comparison morphism $w$ from $X$ to the pullback of that
pushout is a regular epimorphism; since the kernel pairs $\theta$ and
$\varphi$ satisfy $\theta \circ \varphi = \theta \vee \varphi
= \Eq(q_{\theta\vee\varphi})$ (Lemma~\ref{lem:sup}),
Theorem~\ref{thm:ckp57}(a) identifies that pushout with the square
$u \circ q_\theta = v \circ q_\varphi$ of vertex $X/(\theta\vee\varphi)$,
so that $w$ is the canonical morphism
$c = \langle q_\theta, q_\varphi \rangle \colon X \to L(\theta,\varphi)$
and Lemma~\ref{lem:mono} makes $\Phi$ an isomorphism. By
Lemma~\ref{lem:calcul}(3) the two sides of the relational identity $q_\varphi \circ q_\theta^{\op}
\;=\; v^{\op} \circ u$
are -- in the notations of Lemma \ref{lem:mono} --
$\im (c)$ and $L(\theta,\varphi)$, respectively.
\end{proof}

\begin{remarque}\label{rem:reciproque}
In an exact category with pushouts of regular epimorphisms along regular
epimorphisms,
$(\mathrm{CRT}_2)$ conversely implies permutability. The condition $(\mathrm{CRT}_2)$ thus \emph{is} the Mal'tsev
condition: this is Theorem~\ref{thm:ckp57}(b), which in this way
characterizes exact Mal'tsev categories among regular ones. \end{remarque}

\section{Arithmeticity and the CRT for congruences}

\begin{definition}[{cf.\ \cite[Def.~3.12]{Bourn}}]\label{def:arith}
An exact Mal'tsev category $\C$ is \emph{arithmetical} if, for every object
$X$, the lattice $\Equiv(X)$ is distributive.\footnote{This is Bourn's
definition \cite[Def.~3.12]{Bourn} of arithmetical category. Pedicchio's
original definition \cite{Pedicchio} additionally assumed the existence of
coequalizers; Bourn showed this hypothesis to be essentially superfluous
\cite[Thm~3.11]{Bourn}.}
\end{definition}

We shall need the following well-known lattice-theoretic fact:

\begin{lemme}\label{lem:treillis}
 A lattice $T$ is distributive 
 if and only if $(\alpha \vee \beta) \wedge (\alpha \vee \gamma) \le \alpha
\vee (\beta \wedge \gamma)$ for all $\alpha, \beta, \gamma \in T$; equivalently, this is the case if and only if 
$\alpha \wedge (\beta \vee \gamma) \le (\alpha \wedge \beta) \vee (\alpha
\wedge \gamma)$.
\end{lemme}


The main theorem of the paper will be assembled, in the next section, from
the two following propositions: the first one shows that arithmeticity implies the CRT for
families of congruences of any finite size, the second one that the ternary CRT
already forces arithmeticity.

\begin{proposition}\label{prop:arith-crt}
Every arithmetical category satisfies $(\mathrm{CRT}_n)$, for every
$n \ge 2$.
\end{proposition}

\begin{proof}
We argue by induction on $n$, the case $n = 2$ being
Proposition~\ref{prop:binaire}. Let $n \ge 3$; set
$\mu := \theta_1 \wedge \dots \wedge \theta_{n-1}$ and, for $i < n$,
$\alpha_i := \theta_i \vee \theta_n$, with canonical morphisms
$f_i \colon X/\mu \to X/\alpha_i$ and $g_i \colon X/\theta_n \to X/\alpha_i$
(one has $\mu \le \alpha_i$ and $\theta_n \le \alpha_i$).

\emph{Step 1.} Since $f_i = q_{\alpha_i} \circ q_\mu^{\op}$ and
$g_i = q_{\alpha_i} \circ q_{\theta_n}^{\op}$ (\S 2.2), we have
\[
g_i^{\op} \circ f_i
= q_{\theta_n} \circ q_{\alpha_i}^{\op} \circ q_{\alpha_i}
\circ q_\mu^{\op}
= q_{\theta_n} \circ \alpha_i \circ q_\mu^{\op}
\qquad \text{(as relations from } X/\mu \text{ to } X/\theta_n\text{)}.
\]

\emph{Step 2.} We show that
$\bigwedge_{i<n} q_{\theta_n} \circ \alpha_i \circ q_\mu^{\op}
= q_{\theta_n} \circ \bigl(\bigwedge_{i<n} \alpha_i\bigr)
\circ q_\mu^{\op}$.
The inclusion $\supseteq$ is monotonicity. For $\subseteq$, set
$R := \bigwedge_{i<n} q_{\theta_n} \circ \alpha_i \circ q_\mu^{\op}$. Then, by
Lemma~\ref{lem:calcul}(2) and absorption,
\[
q_{\theta_n}^{\op} \circ R \circ q_\mu
\;\le\;
q_{\theta_n}^{\op} \circ \bigl(q_{\theta_n} \circ \alpha_i
\circ q_\mu^{\op}\bigr) \circ q_\mu
= \theta_n \circ \alpha_i \circ \mu = \alpha_i
\quad \text{for every } i<n,
\]
so $q_{\theta_n}^{\op} \circ R \circ q_\mu \le
\bigwedge_{i <n} \alpha_i$, and
\[
R = (q_{\theta_n} \circ q_{\theta_n}^{\op}) \circ R
\circ (q_\mu \circ q_\mu^{\op})
= q_{\theta_n} \circ \bigl(q_{\theta_n}^{\op} \circ R
\circ q_\mu\bigr) \circ q_\mu^{\op}
\;\le\; q_{\theta_n} \circ \Bigl(\bigwedge_{i<n} \alpha_i\Bigr)
\circ q_\mu^{\op}.
\]

\emph{Step 3.} Distributivity gives
\[
\bigwedge_{i<n} \alpha_i
= \bigwedge_{i<n} (\theta_i \vee \theta_n)
= \Bigl(\bigwedge_{i<n} \theta_i\Bigr) \vee \theta_n
= \mu \vee \theta_n,
\]
--- it is here, and only here, that it intervenes. The relation 
of Step 2
thus becomes $q_{\theta_n} \circ (\mu \vee \theta_n) \circ q_\mu^{\op}$.
To identify the corresponding subobject of $X/\mu \times X/\theta_n$,
write $s \colon X/\mu \to X/(\mu \vee \theta_n)$ and
$t \colon X/\theta_n \to X/(\mu \vee \theta_n)$ for the canonical
morphisms, so that $s = q_{\mu\vee\theta_n} \circ q_\mu^{\op}$ and
$t = q_{\mu\vee\theta_n} \circ q_{\theta_n}^{\op}$; then
\[
t^{\op} \circ s
= q_{\theta_n} \circ \bigl(q_{\mu\vee\theta_n}^{\op} \circ
q_{\mu\vee\theta_n}\bigr) \circ q_\mu^{\op}
= q_{\theta_n} \circ (\mu \vee \theta_n) \circ q_\mu^{\op},
\]
and by Lemma~\ref{lem:calcul}(3) the relation $t^{\op} \circ s$ is the
subobject of $X/\mu \times X/\theta_n$ determined by the pullback of $s$
and $t$. The relation of Step 2 is therefore exactly the
pullback $P$ of $X/\mu \to X/(\mu\vee\theta_n) \leftarrow X/\theta_n$.

\emph{Step 4.} We now identify $\Lambda := L(\theta_1,\dots,\theta_n)$
with the pullback $P$. Write $p_i \colon X/\mu \to X/\theta_i$ and
$u_i \colon X/\theta_i \to X/\alpha_i$ ($i < n$) for the canonical
morphisms; then $u_i \circ p_i = f_i$ (both sides composing with the epimorphism $q_\mu$
to $q_{\alpha_i}$). A cone with vertex $T$ over the diagram defining
$\Lambda$ is a family of morphisms $b_i \colon T \to X/\theta_i$
($i \le n$) matching over the $X/(\theta_i \vee \theta_j)$. By the
induction hypothesis the family $(\theta_1,\dots,\theta_{n-1})$ satisfies
the CRT, so that the comparison morphism
$\Phi' \colon X/\mu \to L(\theta_1,\dots,\theta_{n-1})$ is an isomorphism
(Lemma~\ref{lem:mono}); its $i$-th component is $p_i$, both composing
with the epimorphism $q_\mu$ to $q_{\theta_i}$. The matching conditions of indices
$(i,j)$ with $j < n$ therefore hold if and only if $(b_1, \dots, b_{n-1})$ factors,
uniquely, as $b_i = p_i \circ b$ for a morphism $b \colon T \to X/\mu$;
the remaining conditions, of indices $(i,n)$, then read
$f_i \circ b = g_i \circ b_n$. By Lemma~\ref{lem:calcul}(3),
$g_i^{\op} \circ f_i$ is the subobject of $X/\mu \times X/\theta_n$
determined by the pullback of $f_i$ and $g_i$, so these conditions hold if
and only if $\langle b, b_n \rangle$ factors through
$\bigwedge_{i<n} g_i^{\op} \circ f_i$, which is $P$ by Steps 1--3. Cones
over the diagram defining $\Lambda$ therefore correspond bijectively,
naturally in $T$, to cones over the cospan
$X/\mu \to X/(\mu \vee \theta_n) \leftarrow X/\theta_n$, whence
$\Lambda \cong P$. Under this isomorphism a cone $(b_1, \dots, b_n)$
corresponds to the pair $(b, b_n)$, where $b$ is the unique morphism with
$\Phi' \circ b = \langle b_1, \dots, b_{n-1} \rangle$; since
$\Phi' \circ q_\mu = \langle q_{\theta_i} \rangle_{i<n}$ by the 
definition of the comparison morphism $\Phi'$, the canonical morphism
$c \colon X \to \Lambda$ here becomes $\langle q_\mu, q_{\theta_n} \rangle$.

\emph{Step 5.} Since $\mathcal C$ is an exact Mal'tsev category, Proposition~\ref{prop:binaire}, applied to the pair
$(\mu, \theta_n)$, says precisely that
$\im \langle q_\mu, q_{\theta_n} \rangle
= q_{\theta_n} \circ q_\mu^{\op}$ is the whole of $P$. The morphism $c$ is therefore a regular epimorphism, as desired.
\end{proof}

\begin{proposition}\label{prop:crt-arith}
Let $\C$ be an exact Mal'tsev category satisfying $(\mathrm{CRT}_3)$.
Then $\C$ is arithmetical.
\end{proposition}

\begin{proof}
Let $\alpha, \beta, \gamma \in \Equiv(X)$, and set
$r := \langle q_\beta, q_\alpha\rangle \colon X \to X/\beta \times
X/\alpha$. The canonical morphism of the triple $(\beta, \alpha, \gamma)$,
in this order, is $c_3 = \langle r, q_\gamma\rangle$, and $\im (c_3) = q_\gamma \circ r^{\op}$
(Lemma~\ref{lem:calcul}(3)), viewed as a relation from
$X/\beta \times X/\alpha$ to $X/\gamma$. The diagram defining
$L_3 := L(\beta,\alpha,\gamma)$ is
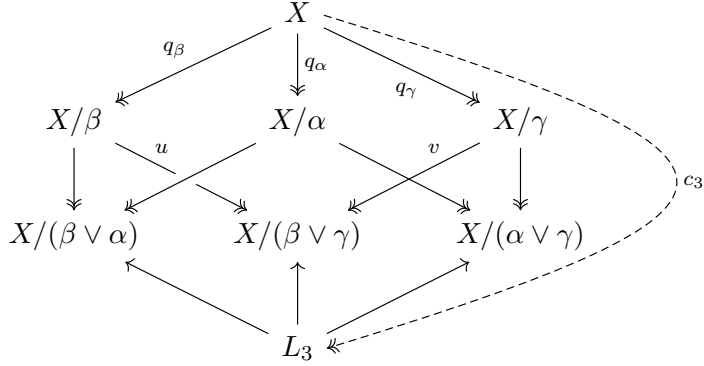
\[
\begin{tikzcd}[column sep=1.2em, row sep=2.2em]
& & X
\arrow[dll, two heads, "q_\beta"']
\arrow[d, two heads, "q_\alpha"]
\arrow[drr, two heads, "q_\gamma"'{pos=0.65}]
\arrow[ddd, dashed, two heads, "c_3",
to path={(\tikztostart.east) .. controls +(6.2,-1.9) and +(6.2,1.9) ..
(\tikztotarget.east)\tikztonodes}] & & \\
X/\beta
\arrow[d, two heads]
\arrow[drr, two heads, "u"{pos=0.25}] & &
X/\alpha
\arrow[dll, two heads, crossing over]
\arrow[drr, two heads, crossing over] & &
X/\gamma
\arrow[dll, two heads, "v"'{pos=0.25}]
\arrow[d, two heads] \\
X/(\beta \vee \alpha) & &
X/(\beta \vee \gamma) & &
X/(\alpha \vee \gamma) \\
& & L_3
\arrow[ull]
\arrow[u]
\arrow[urr] & &
\end{tikzcd}
\]
where $u \colon X/\beta \to X/(\beta\vee\gamma)$ and
$v \colon X/\gamma \to X/(\beta\vee\gamma)$ denote the canonical
morphisms. By $(\mathrm{CRT}_3)$, applied to
this triple, the image of $c_3$ is the whole of $L_3$:
\[
L_3 \;=\; q_\gamma \circ r^{\op},
\]
as relations from $X/\beta \times X/\alpha$ to $X/\gamma$. To prove the congruence distributivity of $\mathcal C$, we now compute
the relation $q_\gamma^{\op} \circ L_3 \circ r$ on $X$ in two ways.

On the one hand, by Lemma~\ref{lem:calcul}(1)(3) and
Lemma~\ref{lem:sup},
\[
q_\gamma^{\op} \circ (q_\gamma \circ r^{\op}) \circ r
= (q_\gamma^{\op} \circ q_\gamma) \circ (r^{\op} \circ r)
= \gamma \circ (\beta \wedge \alpha)
= \gamma \vee (\alpha \wedge \beta).
\]
On the other hand, $q_\gamma^{\op} \circ L_3 \circ r
= (r \times q_\gamma)^{-1}(L_3)$ by Lemma~\ref{lem:calcul}(4), and this
subobject of $X \times X$ can be described directly by its universal
property. A morphism $\langle x, x'\rangle \colon T \to X \times X$
factors through $(r \times q_\gamma)^{-1}(L_3)$ if and only if
$\langle q_\beta \circ x, q_\alpha \circ x, q_\gamma \circ x'\rangle \colon T \to X/\beta
\times X/\alpha \times X/\gamma$ factors through $L_3$, that is, by the
definition of $L_3$ as a limit, if and only if the three compatibility
conditions
\[
u \circ q_\beta \circ x = v \circ q_\gamma \circ x', \qquad
q_{\alpha \vee \gamma} \circ x = q_{\alpha \vee \gamma} \circ x', \qquad
q_{\beta \vee \alpha} \circ x = q_{\beta \vee \alpha} \circ x
\]
hold (in the middle one, the two canonical morphisms
$X/\alpha \to X/(\alpha\vee\gamma) \leftarrow X/\gamma$ of the diagram have
already been composed with $q_\alpha \circ x$ and $q_\gamma \circ x'$). The third
condition is void, and since $u \circ q_\beta = q_{\beta\vee\gamma}
= v \circ q_\gamma$, the first one reads
$q_{\beta\vee\gamma} \circ x = q_{\beta\vee\gamma} \circ x'$.
The first two conditions thus say precisely that $\langle x, x'\rangle$
factors through the kernel pairs $\beta \vee \gamma$ and
$\alpha \vee \gamma$ of $q_{\beta\vee\gamma}$ and $q_{\alpha\vee\gamma}$,
so that $(r \times q_\gamma)^{-1}(L_3) = (\beta\vee\gamma) \wedge
(\alpha\vee\gamma)$. Hence
\[
(\gamma \vee \beta) \wedge (\gamma \vee \alpha)
= q_\gamma^{\op} \circ L_3 \circ r
= \gamma \vee (\alpha \wedge \beta)
\]
for all $\alpha, \beta, \gamma$: this is the distributive law, and
$\Equiv(X)$ is distributive, proving that $\C$ is
arithmetical.
\end{proof}

\section{The characterization via pushouts and limits}\label{sec:diag}

Exactness allows one to eliminate from the statement of the CRT any mention
of equivalence relations and their joins, which are replaced by regular epimorphisms and
pushouts, respectively. The CRT is then a property of regular
epimorphisms, pushouts and finite limits alone:

\begin{proposition}\label{prop:diag}
Let $\C$ be an exact Mal'tsev category and
$r_i \colon X \twoheadrightarrow B_i$ ($1 \le i \le n$) regular
epimorphisms with common domain and kernel pairs $\theta_i$. Then the
pairwise pushouts $B_i \to D_{ij} \leftarrow B_j$ of $(r_i, r_j)$ exist,
the limit of the diagram they form is canonically isomorphic to
$L(\theta_1,\dots,\theta_n)$, and the canonical morphism
$X \to \lim\,(B_i \to D_{ij} \leftarrow B_j)_{i<j}$ is identified with $c$.
In particular, the family $(\theta_i)$ satisfies the CRT if and only if
this morphism is a regular epimorphism.
\end{proposition}

\begin{proof}
Exactness identifies $r_i$ with $q_{\theta_i}$ under $X$. By
Proposition~\ref{prop:binaire}, the
canonical square with vertex $X/(\theta_i \vee \theta_j)$ is a pushout of
$(q_{\theta_i}, q_{\theta_j})$; the diagram of pairwise pushouts is
therefore isomorphic, under the $B_i$, to the diagram defining
$L(\theta_1,\dots,\theta_n)$, and the limits correspond.
\end{proof}

For $n = 3$ the condition is the one displayed in the Introduction. For
$n = 2$ it is exactly the condition of
Theorem~\ref{thm:ckp57}(b) --- the limit being the pullback of the
pushout --- so the CRT for $n$-tuples appears as the 
$n$-th rung of a ladder which starts, at 
$n=2$, with Theorem~5.7 of \cite{CKP}.

This suggests a formulation which makes sense in any regular category
$\mathcal{A}$, where quotients of equivalence relations need not exist: let us say
that $\mathcal{A}$ \emph{satisfies $(\mathrm{CRT}_n)$ for regular
epimorphisms} if, for every family of $n$ regular epimorphisms with common
domain, the pairwise pushouts exist and the comparison morphism to the
limit of their diagram is a regular epimorphism. For $n = 2$ this is
precisely the condition of Theorem~\ref{thm:ckp57}(b) characterizing exact
Mal'tsev categories among regular ones and, by
Proposition~\ref{prop:diag}, when $\mathcal{A}$ is an exact Mal'tsev
category it is equivalent, for every $n$, to $(\mathrm{CRT}_n)$ in the
sense of Definition~\ref{def:trc}. We can now state the main result of
this section, which characterizes arithmetical categories among regular
ones.

\begin{theoreme}\label{thm:main}
Let $\mathcal{A}$ be a regular category. The following conditions are
equivalent:
\begin{itemize}
\item[(i)] $\mathcal{A}$ is an exact arithmetical category;
\item[(ii)] $\mathcal{A}$ satisfies $(\mathrm{CRT}_2)$ and
$(\mathrm{CRT}_3)$ for regular epimorphisms;
\item[(iii)] $\mathcal{A}$ satisfies $(\mathrm{CRT}_n)$ for regular
epimorphisms, for every $n \ge 2$.
\end{itemize}
\end{theoreme}

\begin{proof}
(i) $\Rightarrow$ (iii). An arithmetical category is in particular an
exact Mal'tsev category, and it satisfies $(\mathrm{CRT}_n)$ for every
$n \ge 2$ by Proposition~\ref{prop:arith-crt};
Proposition~\ref{prop:diag} translates this into the condition for regular
epimorphisms.

(iii) $\Rightarrow$ (ii). Trivial.

(ii) $\Rightarrow$ (i). The condition $(\mathrm{CRT}_2)$ for regular
epimorphisms is exactly that of Theorem~\ref{thm:ckp57}(b):
$\mathcal{A}$ is an exact Mal'tsev category.
Proposition~\ref{prop:diag} then translates $(\mathrm{CRT}_3)$ for
regular epimorphisms into $(\mathrm{CRT}_3)$ in the sense of
Definition~\ref{def:trc}, and Proposition~\ref{prop:crt-arith} gives the
distributivity of the lattices $\Equiv(X)$, hence $\mathcal{A}$ is
arithmetical.
\end{proof}

At the level of congruences, Theorem~\ref{thm:main} immediately gives
back the Chinese-remainder characterization of arithmeticity among exact
Mal'tsev categories.

\begin{theoreme}\label{thm:principal}
Let $\C$ be an exact Mal'tsev category. The following conditions are
equivalent:
\begin{itemize}
\item[(i)] $\C$ is arithmetical;
\item[(ii)] $\C$ satisfies $(\mathrm{CRT}_3)$;
\item[(iii)] $\C$ satisfies $(\mathrm{CRT}_n)$ for every $n \ge 2$.
\end{itemize}
\end{theoreme}

\begin{proof}
Being an exact Mal'tsev category, $\C$ satisfies $(\mathrm{CRT}_2)$ for
regular epimorphisms (Theorem~\ref{thm:ckp57}(b)) and, by
Proposition~\ref{prop:diag}, it satisfies $(\mathrm{CRT}_n)$ in the sense
of Definition~\ref{def:trc} if and only if it satisfies
$(\mathrm{CRT}_n)$ for regular epimorphisms. The result follows from Theorem~\ref{thm:main}.\end{proof}

The hypothesis $(\mathrm{CRT}_2)$ in condition (ii) of
Theorem~\ref{thm:main} can in fact be dropped: triples alone suffice.
This is the sharpest form of the characterization, the one stated in the
Introduction.

\begin{theoreme}\label{cor:diag}
A regular category $\mathcal{A}$ is an exact arithmetical category
if and only if, for every triple of regular epimorphisms with common
domain, the pairwise pushouts exist and the comparison morphism to the
limit of their diagram is a regular epimorphism.
\end{theoreme}

\begin{proof}
Necessity is the implication (i) $\Rightarrow$ (iii) of
Theorem~\ref{thm:main}. For sufficiency, we show that the condition on
triples implies the corresponding condition on \emph{pairs}, that is,
$(\mathrm{CRT}_2)$ for regular epimorphisms; Theorem~\ref{thm:main} then
applies.

Let $r \colon X \twoheadrightarrow B$ and
$s \colon X \twoheadrightarrow B'$ be regular epimorphisms, and let
$i_X \circ \sigma_X = \tau_X $ be the (regular epi, mono) factorization
of the
terminal morphism $\tau_X$, with image the \emph{support}
$S \rightarrowtail 1$. One has that
$\Eq(\sigma_X) = \Eq(\tau_X) = \nabla_X$. The morphism $\sigma_B \circ r$ is a regular epi, and
$B$ and $X$ have the same support. Moreover, 
$\sigma_X = \sigma_B \circ r$ and the
square
\[
\begin{tikzcd}
X \arrow[r, two heads, "r"] \arrow[d, two heads, "\sigma_X"'] &
B \arrow[d, two heads, "\sigma_B"] \\
S \arrow[r, equal] & S
\end{tikzcd}
\]
is a pushout.
Now apply the condition $(\mathrm{CRT}_3)$ to the triple $(r, s, \sigma_X)$: the pushouts
with $\sigma_X$ are the two squares just described, both with vertex $S$,
and the pushout $D$ of $(r,s)$ exists by assumption.  Since
$\sigma_B \circ r = \sigma_X = \sigma_{B'} \circ s$, there is a unique
$d \colon D \to S$ with $d \circ \bar r = \sigma_{B'}$ and
$d \circ \bar s = \sigma_B$, where $\bar r \colon B' \to D$ and
$\bar s \colon B \to D$ are the canonical morphisms into the pushout $D$.

Let $L_3$ be the limit of the diagram attached to the triple, and
$B \times_D B'$ the pullback of $\overline{r}$ and $\overline{s}$.
Since $\sigma_B = d \circ \bar s$ and $\sigma_{B'} = d \circ \bar r$, the
two matching conditions involving $S$ are consequences of the one involving
$D$, so that cones over the diagram are just cones over
$B \to D \leftarrow B'$, hence there is a natural isomorphism
$L_3 \cong B \times_D B'$. Under this
isomorphism the comparison morphism $X \to L_3$, whose components are
$(r, s, \sigma_X)$, becomes the comparison morphism $X \to B \times_D B'$ with
components $(r,s)$ --- the third being redundant, precisely because
$\sigma_X = \sigma_B \circ r$. The condition for triples thus gives back
exactly the condition for pairs.

\end{proof}

We then get the following 

\begin{corollaire}\label{cor:nary}
Let $\mathcal{A}$ be a regular category and let $n \ge 3$. Then
$\mathcal{A}$ is an exact arithmetical category if and only if,
for every $n$-tuple of regular epimorphisms with common domain, the
pairwise pushouts exist and the comparison morphism to the limit of their
diagram is a regular epimorphism.
\end{corollaire}

\begin{proof}
Necessity is the implication (i) $\Rightarrow$ (iii) of
Theorem~\ref{thm:main}. Sufficiency rests on a single observation:
\emph{for any family $(r_1, \dots, r_k)$ of regular epimorphisms out of
$X$, with $k \ge 2$, appending the support quotient $\sigma_X$ changes
neither the existence of the pairwise pushouts, nor the limit of the
diagram they form, nor the comparison morphism.}

Indeed, write $r_i \colon X \twoheadrightarrow B_i$ and, as in the proof
of Theorem~\ref{cor:diag}, $\sigma_{B_i} \colon B_i \to S$ for the unique
morphism with $\sigma_{B_i} \circ r_i = \sigma_X$. The pushout of $\sigma_X$ with each $r_i$ is
the square with vertex $S$ and legs $\sigma_{B_i}$ and $1_S$, so the
diagram of the enlarged family exists if and only if that of
$(r_1, \dots, r_k)$ does, all the new vertices being equal to $S$. A cone
over the enlarged diagram is then a cone $(b_1, \dots, b_k)$ over the
original diagram together with a morphism $t \colon T \to S$ subject to
the matching conditions $t = \sigma_{B_i} \circ b_i$ ($1 \le i \le k$).
These conditions determine $t$ and impose nothing on
$(b_1, \dots, b_k)$: they are mutually consistent for every cone, since
the matching condition over the pushouts $D_{ij}$, composed with the canonical
morphism $D_{ij} \to S$ induced by the cocone
$(\sigma_{B_i}, \sigma_{B_j})$, gives
$\sigma_{B_i} \circ b_i = \sigma_{B_j} \circ b_j$. Cones over the two
diagrams therefore correspond bijectively, naturally in the vertex $T$:
the limits are canonically isomorphic, and the comparison morphisms are
identified, the new component of the enlarged one being
$\sigma_X = \sigma_{B_i} \circ r_i$.

Now let $m \in \{2, 3\}$ and let $(r_1, \dots, r_m)$ be regular
epimorphisms with common domain $X$. Padding with $n - m$ copies of
$\sigma_X$ yields an $n$-tuple, whose comparison morphism is a regular
epimorphism by hypothesis; removing the copies of $\sigma_X$ one at a
time, the observation identifies it with the comparison morphism of
$(r_1, \dots, r_m)$, which is therefore a regular epimorphism as well.
Thus $\mathcal{A}$ satisfies $(\mathrm{CRT}_2)$ and $(\mathrm{CRT}_3)$
for regular epimorphisms, and Theorem~\ref{thm:main} applies.
\end{proof}
We conclude this section with the following simple observations:

\begin{remarque}[The Mal'tsev assumption is necessary already for $n=2$]
\label{ex:set}
Definition~\ref{def:trc} makes sense in any exact category with pushouts
of regular epimorphisms along regular epimorphisms (\S 3.1), which allows
us to step outside the Mal'tsev
setting for a moment and to consider the exact category $\Set$ of sets, and reason
with elements. Let $X = \{1,2,3\}$, let $\theta$ be the congruence
with classes $\{1,2\},\{3\}$ and $\varphi$ the one with classes
$\{1\},\{2,3\}$. Then $\theta \wedge \varphi = \Delta$ and
$\theta \vee \varphi = \nabla$. The pullback $L(\theta,\varphi) =
X/\theta \times X/\varphi$ has $4$ elements while $X/\Delta = X$ has $3$:
the family $([3]_\theta, [1]_\varphi)$ is  compatible, but it has no
solution.
\end{remarque}

\begin{remarque}[The Mal'tsev assumption does not suffice from $n=3$ on]\label{ex:m3}
It is well known that the lattice of equivalence relations on any object $X$ in an exact Mal'tsev category is modular, but not necessarily distributive. 
For instance, the variety of abelian groups 
then satisfies $(\mathrm{CRT}_2)$ but not $(\mathrm{CRT}_3)$.
\end{remarque}

\begin{remarque}[comaximality]
If the $\theta_i$ are pairwise \emph{comaximal}
($\theta_i \vee \theta_j = \nabla$), the compatibility conditions are
vacuous, $L(\theta_1,\dots,\theta_n) = \prod_i X/\theta_i$, and the CRT
reduces to the decomposition
$X/\bigwedge_i \theta_i \cong \prod_i X/\theta_i$. This very restricted form can
hold \emph{without} distributivity: in the variety of unital commutative rings --- which
is not congruence-distributive --- comaximality is preserved by products of ideals,
since $I + J = I + K = R$ gives
\[
R = R \cdot R = (I+J)(I+K) \subseteq I + JK ,
\]
whence $I + JK = R$; the classical \emph{Chinese Remainder Theorem} for rings
then follows by induction. The argument is about products of ideals, not
about the lattice of congruences. There is therefore no contradiction
between the classical CRT for rings and Theorem~\ref{thm:principal}: what
distributivity characterizes is the solvability of compatible systems
\emph{without any comaximality assumption}.
\end{remarque}

\section{A categorical proof of Pixley's theorem}\label{sec:pixley}

In the varietal case, the diagram of Section~\ref{sec:diag}, instantiated
at a free algebra, yields a proof of Pixley's theorem \cite{Pixley} in
which the characteristic ternary term is deduced from a single, generic instance
of the CRT. Write
$X = F(1)$ for the free algebra on one generator, so that the free algebra
on three generators is the copower $F(x,y,z) = X + X + X$.

\begin{theoreme}[Pixley \cite{Pixley}]\label{thm:pixley}
For a variety $\mathbb{V}$ of universal algebras, the following conditions
are equivalent:
\begin{itemize}
\item[(1)] $\mathbb{V}$ is arithmetical, i.e.\ congruence-permutable and
congruence-distributive;
\item[(2)] $\mathbb{V}$ satisfies $(\mathrm{CRT}_n)$ for every $n \ge 2$;
\item[(3)] the single triple of congruences
$\theta_1 = \mathrm{Cg}(x,y)$, $\theta_2 = \mathrm{Cg}(x,z)$,
$\theta_3 = \mathrm{Cg}(y,z)$ on $F(x,y,z)$ --- each identifying two of
the three generators --- satisfies the CRT;
\item[(4)] the theory of $\mathbb{V}$ contains a ternary operation $p$
satisfying the identities
\[
p(x,x,z) = z, \qquad p(x,y,x) = x, \qquad p(x,y,y) = x.
\]
\end{itemize}
\end{theoreme}

\begin{proof}
(1) $\Rightarrow$ (2). A variety is an exact category, its internal
equivalence relations are the usual congruences, and congruence-permutability is the
Mal'tsev condition (Theorem~\ref{thm:ckp}), hence Theorem~\ref{thm:principal}
applies.

(2) $\Rightarrow$ (3). Trivial. Note that condition (3) makes sense in any
variety, congruence permutable or not: a variety is exact and, being cocomplete, has
all pushouts.
The implication (3) $\Rightarrow$ (4) below uses Lemma~\ref{lem:mono} in a
variety not yet known to be congruence-permutable --- deriving
permutability is part of what is to be proved --- and Lemma~\ref{lem:mono}
is available there, its proof using only  exactness and
Lemma~\ref{lem:calcul}.

(3) $\Rightarrow$ (4). The quotients of $F(x,y,z) = X + X + X$ by the three
congruences are the homomorphisms 
\[
q_1 = \nabla_2 + 1, \qquad
q_2 = \bigl[\, i_1,\ i_2,\ i_1 \,\bigr], \qquad
q_3 = 1 + \nabla_2
\ \colon\ X + X + X \longrightarrow X + X ,
\]
identifying $x$ with $y$, $x$ with $z$, and $y$ with $z$ respectively, where $i_1$ and $i_2$ are the coproduct injections, and $\nabla_2 = [\,1, 1\,] \colon X + X \to X$ the codiagonal.
The morphisms $q_1, q_2 $ and $q_3$ are split epimorphisms. Moreover, any two of the
three congruences already identify all three generators, hence the ``object part'' of the pushouts of $q_i$ and $q_j$ (for $i \not=j$) is the free algebra $X = F(1)$ on one
generator. Indeed, every join $\theta_i \vee \theta_j$ is the kernel pair of the
codiagonal $\nabla_3 = [\,1, 1, 1\,] \colon X + X + X \to X$, so that
$F(x,y,z)/(\theta_i \vee \theta_j) = X$ for each of the three pairs. The
diagram \eqref{cube2} attached to this triple thus takes the form
\begin{equation}\label{Free}
\begin{tikzcd}[column sep=1.2em, row sep=2.2em]
& & X + X + X
\arrow[dll, two heads, "q_1"']
\arrow[d, two heads, "q_2"]
\arrow[drr, two heads, "q_3"'{pos=0.65}]
\arrow[ddd, dashed, two heads, "c",
to path={(\tikztostart.east) .. controls +(6.2,-1.9) and +(6.2,1.9) ..
(\tikztotarget.east)\tikztonodes}] & & \\
X + X
\arrow[d, two heads, "\nabla_2"']
\arrow[drr, two heads] & &
X + X
\arrow[dll, two heads, crossing over]
\arrow[drr, two heads, crossing over] & &
X + X
\arrow[dll, two heads]
\arrow[d, two heads, "\nabla_2"] \\
X & & X & & X \\
& & L(\theta_1, \theta_2, \theta_3)
\arrow[ull]
\arrow[u]
\arrow[urr] & &
\end{tikzcd}
\end{equation}
 The triple of elements
\[
\bigl(\, z \md{\theta_1},\ \ x \md{\theta_2},\ \ x \md{\theta_3} \,\bigr)
\]
is pairwise compatible: $(x,z) \in \theta_2 \le \theta_1 \vee \theta_2$
takes care of the first pair, $(x,z) \in \theta_3 \circ \theta_1 \le
\theta_1 \vee \theta_3$ (through $y$) of the second, and the third is
trivial. Being an element of $L(\theta_1,\theta_2,\theta_3) $ and $X=F(1)$ being free on one generator, it therefore defines a morphism
$X \to L(\theta_1,\theta_2,\theta_3)$. Since the free algebra $X =F(1)$ is
projective with respect to regular epimorphisms, this morphism lifts along
the canonical $c \colon F(x,y,z) \to L(\theta_1,\theta_2,\theta_3)$, which
is a regular epimorphism by (3) and Lemma~\ref{lem:mono} (applied in the
setting of \S 3.1, which does not assume congruence permutability). The lift is a morphism
$X = F(1) \to F(x,y,z)$, that is, an element $p$ of $F(x,y,z)$, that is
again, a ternary term $p(x,y,z)$; it satisfies
\[
p \equiv z \md{\theta_1}, \qquad
p \equiv x \md{\theta_2}, \qquad
p \equiv x \md{\theta_3}.
\]
Since the quotients $F(x,y,z)/\theta_i \cong X + X$ are themselves free,
equalities of elements there are identities of the algebraic theory of $\mathbb{V}$, so that 
\[
p(x,x,z) = z, \qquad p(x,y,x) = x, \qquad p(x,y,y) = x ,
\]
are identities: this is exactly (4).

(4) $\Rightarrow$ (1). The two outer identities make $p$ a Mal'tsev term,
so congruences permute. For distributivity, let $\alpha, \beta, \gamma$ be
congruences on an algebra $A$ and $(a,b) \in \alpha \wedge (\beta \vee
\gamma)$. By permutability $\beta \vee \gamma = \beta \circ \gamma$, so
there is a $c \in A$ with $a \mathrel{\beta} c \mathrel{\gamma} b$. Set
$d := p(a,c,b)$. Then $d \mathrel{\gamma} p(a,b,b) = a$ and
$d \mathrel{\beta} p(c,c,b) = b$, while $(b,a) \in \alpha$ and the middle
identity give $d \mathrel{\alpha} p(a,c,a) = a$, hence also
$d \mathrel{\alpha} b$. Thus $(a,d) \in \alpha \wedge \gamma$ and
$(d,b) \in \alpha \wedge \beta$, so that
$(a,b) \in (\alpha \wedge \gamma) \circ (\alpha \wedge \beta) \subseteq
(\alpha \wedge \beta) \vee (\alpha \wedge \gamma)$. By
Lemma~\ref{lem:treillis}, the congruence lattices are distributive.
\end{proof}

\begin{remarque}[Comparison with the literature]\label{rem:litterature}
The relationship between Theorem~\ref{thm:principal} and Hoefnagel's work
on majority categories \cite{Hoefnagel, Hoefnagel2, Hoefnagel3} was set out in the Introduction, and we
only add two points here. First, in the exact Mal'tsev context
Proposition~\ref{prop:binaire} identifies pairwise solvability with
pairwise compatibility, via $\theta_i \circ \theta_j = \theta_i \vee
\theta_j$; this is what allows the limit $L(\theta_1,\dots,\theta_n)$ to be
formed from the joins alone, and hence, after
Proposition~\ref{prop:diag}, from pushouts alone. 
On the varietal side, this “pairwise” form of the Chinese
Remainder Theorem goes back to Baker–Pixley \cite{BP}.

Secondly, the Mal'tsev
assumption is fundamental: without it, a majority category is not necessarily distributive even if it is exact, as observed in \cite{Hoefnagel}, by referring to a counter-example due to G. Janelidze. Indeed, this latter showed in \cite{GJanelidze} that there is 
an infinitary variety of distributive lattices satisfying the so-called shifting property \cite{BG} which is not congruence modular. In particular this provides an example of a majority category which is not congruence distributive.

The
categorical approach to (proto)arithmetical categories is also developed by Pedicchio \cite{Pedicchio} and by Bourn \cite{Bourn, Bourn2}. 
To the
best of our knowledge, the diagrammatic characterization of
Theorems~\ref{thm:main} and~\ref{cor:diag} and Corollary~\ref{cor:nary}, and the proof of
Pixley's theorem given in Section~\ref{sec:pixley}, are new.
\end{remarque}

\finaladdress{Marino Gran}%
{Institut de Recherche en Mathématique et Physique, Université catholique de Louvain, B-1348 Louvain-la-Neuve, Belgium}%
{marino.gran@uclouvain.be}

\end{document}